\documentclass[a4paper, 11pt]{amsart}
\usepackage[latin1]{inputenc}
\usepackage[ T1]{fontenc}
\usepackage[english]{babel}
\usepackage{amssymb}
\usepackage{amsmath}
\usepackage{amsthm}
\usepackage{amscd}
\usepackage{amsfonts}
\usepackage{stmaryrd}
\usepackage{pb-diagram}
\usepackage{epic,eepic,epsfig}
\usepackage{a4wide}
\usepackage{nextpage}
\usepackage{fancyhdr}
\usepackage{tikz}

\newtheorem{prop}{Proposition}[section]
\newtheorem{cor}[prop]{Corollary}

\newtheorem{lem}[prop]{Lemma}
\newtheorem{thm}[prop]{Theorem}
\newtheorem{conj}[prop]{Conjecture}

\theoremstyle{definition}
\newtheorem{rem}[prop]{Remark}
\newtheorem{defin}[prop]{Definition}

\newcommand{\Jac} {\mathop{\mathrm{Jac}}}
\newcommand{\Ima} {\mathop{\mathrm{Im}}}

\newcommand{\degr} {\mathop{\mathrm{deg}}}

\newcommand{\Spec} {\mathop{\mathrm{Spec}}}

\newcommand{\rank} {\mathop{\mathrm{rank}}}

\newcommand{\hFplus} {\mathop{h_{\mathrm{F}^+}}}
\newcommand{\Reg} {\mathop{\mathrm{Reg}}}

\newcommand{\End} {\mathop{\mathrm{End}}}

\usepackage[OT2,T1]{fontenc}
\DeclareSymbolFont{cyrletters}{OT2}{wncyr}{m}{n}
\DeclareMathSymbol{\Sha}{\mathalpha}{cyrletters}{"58}

\begin{document}

\title{Heights, ranks and regulators of abelian varieties}
\author{{F}abien {P}azuki}
\address{Department of Mathematical Sciences, University of Copenhagen,
Universitetsparken 5, 
2100 Copenhagen \O, Denmark, and Institut de Math\'ematiques de Bordeaux, Universit\'e de Bordeaux,
351, cours de la Lib\'e\-ra\-tion, 
33405 Talence, France.}
\email{fabien.pazuki@math.u-bordeaux.fr}
\thanks{Many thanks to Pascal Autissier, Marc Hindry, Robin de Jong, Qing Liu, Ga\"el R\'emond and Martin Widmer for exchanging ideas concerning this article. The referees gave good advice leading to several improvements, may they be warmly thanked. The author is supported by the DNRF Niels Bohr Professorship of Lars Hesselholt and ANR-14-CE25-0015 Gardio.}
\maketitle
\vspace{-0.9cm}
\begin{center}

\vspace{0.2cm}
\end{center}

\vspace{0.5cm}

\noindent \textbf{Abstract.}
We give a lower bound on the Faltings height of an abelian variety over a number field by the sum of its injectivity diameter and the norm of its bad reduction primes. It leads to an unconditional explicit upper bound on the rank of Mordell-Weil groups. Assuming the height conjecture of Lang and Silverman, we then obtain a Northcott property for the regulator on the set of polarized simple abelian varieties defined over a fixed number field $K$, of dimension $g$ and rank $m_K$ bounded from above and with dense $K$-rational points. We remove the simplicity assumption in the principally polarized case by giving a refined version of the Lang-Silverman conjecture.

{\flushleft
 }

\begin{center}
---------
\end{center}

\begin{center}
\textbf{Hauteurs, rangs et r\'egulateurs des vari\'et\'es ab\'eliennes}.
\end{center}

\noindent \textbf{R\'esum\'e.}
On minore la hauteur de Faltings d'une vari\'et\'e ab\'elienne sur un corps de nombres par la somme de son diam\`etre d'injectivit\'e et de la norme de ses premiers de mauvaise r\'eduction. Cela entra\^ine une majoration explicite inconditionnelle du rang des groupes de Mordell-Weil. On obtient alors comme cons\'equence d'une conjecture de Lang et Silverman une propri\'et\'e de Northcott pour le r\'egulateur sur l'ensemble des vari\'et\'es ab\'eliennes simples, polaris\'ees et d\'efinies sur un corps de nombres, \`a dimension et rang born\'es et dont les points rationnels sont denses. On montre comment se passer de l'hypoth\`ese de simplicit\'e dans le cas de polarisation principale en proposant une version raffin\'ee de la conjecture de Lang-Silverman.

{\flushleft
\textbf{Mots-Clefs:} Hauteurs, vari\'et\'es ab\'eliennes, r\'egulateurs, Mordell-Weil.\\
}

\begin{center}
---------
\end{center}

\thispagestyle{empty}

\section{Introduction}

Let $K$ be a number field of degree $d$ over $\mathbb{Q}$ and let $M_{K}$ stand for the set of all places of $K$. We denote by $M_{K}^{\infty}$ the set of archimedean places. For any place $v\in{M_K}$, we denote by $K_{v}$ the completion of $K$ with respect to the valuation $|.|_{v}$. One fixes $|p|_{v}=p^{-1}$ as a normalization for any finite place $v$ above a rational prime $p$. The local degree will be denoted by $d_v=[K_v:\mathbb{Q}_v]$. 

Let $A$ be an abelian variety of dimension $g$ defined over the number field $K$. The set of rational points of $A$ over $K$ is finitely generated by the Mordell-Weil Theorem, and the aim of this article is to study links between the rank of the Mordell-Weil group, the regulator of the Mordell-Weil lattice and the Faltings height of the abelian variety $A$.

We first give an inequality between the Faltings height $\hFplus(A/K)$ of Definition \ref{faltings} and the norm of the bad reduction primes of $A$, interesting in itself and useful for the following results.

\begin{thm}\label{height conductor intro} 
Let $g\geq 1$ be an integer. Let $K$ be a number field of degree $d$. There exist two quantities $c=c(g)>0$, $c_0=c_0(g)\in{\mathbb{R}}$, such that for any abelian variety $A$ of dimension $g$, defined over $K$, one has $$\hFplus(A/K) \geq c \frac{1}{d}\log N^{0}_{A/K}+c_0$$ where $N^0_{A/K}$ is the product of the norms of the primes of $K$ of bad reduction for $A$. The explicit values $c=(12g)^{-12g^{12g^{4g}}}$ and $c_0=-1/c$ are valid.
\end{thm}

We believe this inequality will be useful in different contexts. We believe furthermore that some steps in the proof presented here could be useful as well (explicit Bertini, reduction to the jacobian case by quotient with explicit bounds on their heights, etc). See Proposition \ref{semi stable height conductor} for a detailed description of the argument. A similar statement (for the semi-stable case) was obtained in \cite{HiPa16} independently. They show the result over function fields and with different arguments. The main difference is the fact that the quantities $c$ and $c_0$ here don't depend on the base field $K$, but only on the dimension $g$. Their strategy of proof seems to work over number fields as well, based on rigid uniformization of abelian varieties.  Remark that their lower bound (at least in the case of function fields) is given in terms of Tamagawa numbers. Another difference is that the quantities $c$ and $c_0$ here are given explicitly, and are not too extreme in the case of jacobian varieties, where $c=1/12^{4g^2+1}$ and $c_0=0$, see Proposition \ref{semi stable height conductor} and Proposition \ref{height conductor} below. Unfortunately, it is unlikely that the explicit expressions we obtain could be improved in a significant way using the strategy we follow here. The main reason is the use of Theorem 1.3 page 760 of \cite{Remond} where one already has a tower with three levels of exponents, plus the fact that the control given in \cite{CaTa} on the genus of the curve constructed by the Bertini argument is (more than) exponential: the combination of these inequalities is forcing our $c$ and our $c_0$ to behave like a four levels exponential function in the dimension of $A$, and the aforementioned results both seem difficult to improve on.

Whenever a polarization is given on $A$, one can obtain a richer lower bound. Let $(A,L)$ be a polarized abelian variety of dimension $g$ defined over the number field $K$. We give an inequality between the Faltings height $\hFplus(A/K)$ of Definition \ref{faltings}, the norm of the bad reduction primes of $A$ over $K$ and the injectivity diameter of $(A(\mathbb{C}),L)$. As a direct corollary of Theorem \ref{height conductor intro} and of the Matrix Lemma (see Theorem \ref{matrix} below) we obtain:

\begin{cor}\label{height conductor injectivity} 
Let $g\geq 1$ be an integer. Let $K$ be a number field of degree $d$. There exist three explicit quantities $c_1=c_1(g)>0$, $c_2=c_2(g)>0$, $c_3=c_3(g)\in{\mathbb{R}}$, such that for any abelian variety $A$ of dimension $g$, defined over $K$, for any ample line bundle $L$ carrying a polarization on $A$, one has $$\hFplus(A/K) \geq c_1 \frac{1}{d}\log N^{0}_{A/K}+c_2\frac{1}{d}\sum_{v\in{M_K^{\infty}}}d_v\rho(A_v,L_v)^{-2}+c_3$$ where $N^0_{A/K}$ is the product of the norms of the primes of $K$ of  bad reduction for $A$ and $\rho(A_v,L_v)$ is the injectivity diameter of $A_v(\mathbb{C})$ polarized by $L_v$. The explicit values $c_1=(12g)^{-12g^{12g^{4g}}}/17$, $c_2=1/17$ and $c_3=-1/c_1$ are valid.
\end{cor}

Note that the polarization is not required to be principal here. As for Theorem \ref{height conductor intro}, the explicit values given here are not expected to be optimal. Nevertheless in the case where $A$ is the semi-stable jacobian of a curve, one can take $c_1=1/204$, $c_2=1/17$ and $c_3=-39g/17$.

One obtains the following result as another corollary of Theorem \ref{height conductor intro} and preexisting bounds on the Mordell-Weil rank.

\begin{cor}\label{rank height intro}
Let $A$ be an abelian variety of dimension $g$ defined over a number field $K$ of degree $d$ and discriminant $\Delta_K$. Let $m_K$ be the Mordell-Weil rank of $A(K)$. There exists a quantity $c_{4}=c_{4}(d,g)>0$ such that $$m_K\leq c_{4} \max\{1,\hFplus(A/K), \log\vert \Delta_K\vert\},$$ and the explicit value $c_4=(12g)^{12g^{12g^{4g}}}d^3$ is valid.
\end{cor}

Let us add another remark, namely if $A=J_C$ is the jacobian of a curve $C$ of genus $g\geq 2$ (not necessarily semi-stable) over a number field $K$ of degree $d$ and discriminant $\Delta_K$, one has the explicit $$m_K\leq 48g^3d^32^{8g^2}12^{4g^2}\max\{1,\hFplus(J_C/K)\} +gd 2^{8g^2} \log\vert \Delta_{K}\vert + g^3d^32^{8g^2}\log16,$$ as given in the proof of Corollary \ref{rank height intro}. The case of elliptic curves is given in Lemma 4.7 of \cite{Paz14}. Corollary \ref{rank height intro} will be used in the proof of Lemma \ref{success}.

We then focus on the regulator of $A(K)$. We show that it verifies a Northcott property for simple abelian varieties under a conjecture of Lang and Silverman, as proposed in \cite{Paz16}.\footnote{Note that a version of Theorem \ref{reg height thm intro simple} for elliptic curves without the requirement that the rank is bounded from above is given in \cite{Paz14} with an incorrect proof, see \cite{Paz16b}.}

\begin{thm}\label{reg height thm intro simple}
Assume the Lang-Silverman Conjecture \ref{LangSilV1}.
The set of $\overline{\mathbb{Q}}$-isomorphism classes of simple abelian varieties $A$, equipped with an ample and symmetric line bundle $L$, defined over a fixed number field $K$, of dimension $g$ and rank $m_K$ bounded from above, with $A(K)$ Zariski dense in $A$ and with $\Reg_L(A/K)$ bounded from above is finite.
\end{thm}

In this special case one restricts to simple abelian varieties where the Zariski density of $A(K)$ is equivalent to having positive Mordell-Weil rank. Using a stronger height conjecture one obtains the following general statement, without any simplicity assumption, on the moduli space of principally polarized abelian varieties.

\begin{thm}\label{reg height thm intro}
Assume the stronger Lang-Silverman Conjecture \ref{LangSilV2}.
The set of $\overline{\mathbb{Q}}$-isomorphism classes of principally polarized abelian varieties $A$, defined over a fixed number field $K$, of dimension $g$ and rank $m_K$ bounded from above, with $A(K)$ Zariski dense in $A$ and regulator bounded from above is finite.
\end{thm}

As explained in \cite{Paz16}, if one restricts to $g=1$ one can replace the Lang-Silverman conjecture by the ABC conjecture in the statements of Theorem \ref{reg height thm intro simple} and Theorem \ref{reg height thm intro}. 

We divide the work as follows. In section \ref{section def} we give the definitions of the regulator and of the Faltings height of an abelian variety. In section \ref{section lower} we prove Theorem \ref{height conductor intro}: it relies on the core of the work, Proposition \ref{semi stable height conductor}, which gives the semi-stable version. The final step is then given in Proposition \ref{height conductor}. In section \ref{section LS} we use the conjecture of Lang and Silverman to deduce Theorem \ref{reg height thm intro simple}. In section \ref{section LS strong} we discuss how a stronger conjecture of Lang and Silverman type imply Theorem \ref{reg height thm intro}. We conclude in section \ref{section conclusion} with a comparaison with number fields, extending the dictionnary of \cite{Paz14}.

\section{Definitions}\label{section def}

Let $S$ be a set. We will say that a function $f: S\to \mathbb{R}$ verifies a Northcott property on $S$ if for any real number $B$, the set $\{P\in{S}\,\vert\, f(P)\leq B\}$ is finite. 

Notation: the function denoted $\log$ is the reciprocal of the classical exponential function, so $\log e=1$ (we do not use the notation $\mathrm{ln}$). 
We will denote by $\mathcal{O}_K$ the ring of integers of $K$. If $K'$ is a finite extension of a number field $K$, we denote by $\mathcal{N}_{K'/K}$ the relative norm. If $\mathfrak{p}'$ is a prime ideal in $\mathcal{O}_{K'}$ above the prime ideal $\mathfrak{p}$ in $\mathcal{O}_K$, then $e_{\mathfrak{p}'/\mathfrak{p}}$ is the ramification index and $f_{\mathfrak{p}'/\mathfrak{p}}$ stands for the residual degree.

\subsection{Regulators of abelian varieties}

Let $A/K$ be an abelian variety over a number field $K$ polarized by an ample and symmetric line bundle $L$. Let $m_K$ be the Mordell-Weil rank of $A(K)$. Let $\widehat{h}_{A,L}$ be the N\'eron-Tate height associated with the pair $(A,L)$. Let $<.,.>$ be the associated bilinear form, given by $$<P,Q>=\frac{1}{2}\Big(\widehat{h}_{A,L}(P+Q)-\widehat{h}_{A,L}(P)-\widehat{h}_{A,L}(Q)\Big)$$ for any $P,Q\in{A(\overline{\mathbb{Q}})}$. This pairing on $A\times A$ depends on the choice of $L$.

\begin{defin}\label{reg abvar}
Let $P_1, ..., P_{m_K}$ be a basis of the lattice $A(K)/A(K)_{\mathrm{tors}}$, where $A(K)$ is the Mordell-Weil group. The regulator of $A/K$ is defined by $$\mathrm{Reg}_{L}(A/K)=\vert \det(<P_i,P_j>_{1\leq i,j\leq m_K})\vert,$$ where by convention an empty determinant is equal to $1$.
\end{defin}

As for the height, the regulator of an abelian variety depends on the choice of an ample and symmetric line bundle $L$ on $A$. 

There is a more intrinsic way of defining a regulator, that doesn't depend on the choice of $L$. Start with the natural pairing on the product of $A$ with its dual abelian variety $\check{A}$ given by the Poincar\'e line bundle $\mathcal{P}$: for any $P\in{A(\overline{\mathbb{Q}})}$ and any $Q\in{\check{A}(\overline{\mathbb{Q}})}$, define $<P,Q>=\widehat{h}_{A\times \check{A},\mathcal{P}}(P,Q)$. Next choose a base $P_1, ..., P_{m_K}$ of $A(K)$ modulo torsion and a base $Q_1, ..., Q_{m_K}$ of $\check{A}(K)$ modulo torsion. Then define $$\Reg(A/K)=\vert \det(<P_i, Q_j>_{1\leq i,j\leq m_K})\vert.$$

Let us recall how these two regulators are linked (see for instance \cite{Hin} page 172). Let $\Phi_L:A\to \check{A}$ be the isogeny given by $\Phi_L(P)=t_{P}^{*}L\otimes L^{-1}$. One has the formula $$\widehat{h}_{A,L}(P)=-\frac{1}{2}<P,\Phi_L(P)>.$$ 
Hence if $u$ is the index of the subgroup $\Phi_L(\mathbb{Z}P_1\oplus ... \oplus\mathbb{Z}P_{m_K})$ in $\check{A}(K)$ modulo torsion, one has 
\begin{equation}\label{regulators}
\mathrm{Reg}_L(A/K)=u2^{-m_K}\mathrm{Reg}(A/K).
\end{equation}

\noindent Let us remark that when $L$ induces a principal polarization, the index $u$ is equal to $1$. Thus Theorem \ref{reg height thm intro} is valid with both regulators.

\subsection{The Faltings height}

Let $A$ be an abelian variety defined over a number field $K$, of dimension $g\geq 1$. Recall that ${\mathcal O}_K$ is the ring of integers of $K$ and let $\pi\colon {\mathcal A}\longrightarrow \Spec(\mathcal{O}_K) $
be the N\'eron model of $A$ over $\Spec(\mathcal{O}_K)$. Let $\varepsilon\colon \Spec(\mathcal{O}_K)\longrightarrow {\mathcal A}$ be the zero section of $\pi$ and let
$\omega_{{\mathcal A}/\mathcal{O}_K}$ be the maximal exterior power (the determinant) of the sheaf of relative differentials
$$\omega_{{\mathcal A}/\mathcal{O}_K}:=\varepsilon^{\star}\Omega^g_{{\mathcal
A}/\mathcal{O}_K}\simeq\pi_{\star}\Omega^g_{{\mathcal A}/\mathcal{O}_K}\;.$$

For any archimedean place $v$ of $K$, let $\sigma$ be an embedding of $K$ in $\mathbb{C}$ associated to $v$. The associated line bundle
$$\omega_{{\mathcal A}/\mathcal{O}_K,\sigma}=\omega_{{\mathcal A}/\mathcal{O}_K}\otimes_{{\mathcal O}_K,\sigma}\mathbb{C}\simeq H^0({\mathcal
A}_{\sigma}(\mathbb{C}),\Omega^g_{{\mathcal A}_\sigma}(\mathbb{C}))\;$$
is equipped with a natural $L^2$-metric $\Vert.\Vert_{v}$ given by
$$\Vert s\Vert_{v}^2=\frac{i^{g^2}}{(2\pi)^{2g}}\int_{{\mathcal
A}_{\sigma}(\mathbb{C})}s\wedge\overline{s}\;.$$

The ${\mathcal O}_K$-module $\omega_{{\mathcal A}/\mathcal{O}_K}$ is of rank $1$ and together with the hermitian norms
$\Vert.\Vert_{v}$ at infinity it defines an hermitian line bundle 
$\overline{\omega}_{{\mathcal A}/\mathcal{O}_K}=({\omega}_{{\mathcal A}/\mathcal{O}_K}, (\Vert .\Vert_v)_{v\in{M_K^\infty}})$ over $\mathcal{O}_K$. It has a well defined Arakelov degree
$\widehat{\degr}(\overline{\omega}_{{\mathcal A}/\mathcal{O}_K})$. Recall that for any hermitian line
bundle $\overline{\mathcal L}$ over $\Spec(\mathcal{O}_K)$ the degree of $\overline{\mathcal L}$ in the sense of Arakelov
is defined as
$$\widehat{\degr}(\overline{\mathcal L})=\log\#\left({\mathcal L}/{s{\mathcal
O}}_K\right)-\sum_{v\in{M_{K}^{\infty}}}d_v\log\Vert
s\Vert_{v}\;,$$
where $s$ is any non zero section of $\mathcal L$. The resulting number does not depend on the choice
of $s$ in view of the product formula on the number field $K$. 

The Arakelov degree of this metrized bundle will give a translate of the classical Faltings height.

\begin{defin}\label{faltings}  The height of $A/K$ is defined
as
$$\hFplus(A/K):=\frac{1}{[K:\mathbb{Q}]}\widehat{\degr}(\overline{\omega}_{{\mathcal
A}/\mathcal{O}_K})\;.$$
\end{defin}
This non-negative real number doesn't depend on any choice of polarization on $A$. When $A/K$ is semi-stable, this height only depends on the $\overline{\mathbb{Q}}$-isomorphism class of $A$. It is just a translate of the classical Faltings height $h_{F}(A/K)$, we have $\hFplus(A/K)=h_{F}(A/K)+\frac{g}{2}\log(2\pi^2)$. If $A/K$ is not semi-stable, one may use Chai's base change conductor as in the formula (\ref{base change conductor}) in the sequel as a complementary definition. See \cite{Falt} Satz~1, page 356 and 357 for its basic properties, and for a comparison with the theta height in \cite{Paz3} (based on ideas of Bost and David). We prefer to use this translate because it gives cleaner inequalities (see the jacobian case in Proposition \ref{semi stable height conductor} for instance). We recall here four classical properties: firstly, if $A=A_1\times A_2$ is a product of abelian varieties, one has $\hFplus(A/K)=\hFplus(A_1/K)+\hFplus(A_2/K)$. Secondly, the dual abelian variety of $A$ has the same height as $A$ by a result of Raynaud. Thirdly, if $K'/K$ is a number field extension, one has $\hFplus(A/K')\leq \hFplus(A/K)$. Finally if $A/K$ is semi-stable, one defines the stable height by $\hFplus(A/\overline{\mathbb{Q}}):=\hFplus(A/K)$, which is invariant by number field extension.

At finite places we focus on the bad reduction locus with the following quantity.

\begin{defin}
Let $A$ be an abelian variety over a number field $K$. Let $\mathcal{A}\to \Spec(\mathcal{O}_K)$ be its N\'eron model. Let $\mathfrak{p}$ be a prime of $\mathcal{O}_K$. If the special fiber $\mathcal{A}_\mathfrak{p}$ is an abelian variety, we say that $\mathfrak{p}$ is a prime of good reduction for $A$, otherwise the prime is of bad reduction. We define $$N_{A/K}^0=\prod_{\substack{\mathfrak{p}\subset{\mathcal{O}_K}, \\ \mathfrak{p}\,\mathrm{bad}\, \mathrm{for}\, A}} \mathcal{N}_{K/\mathbb{Q}}(\mathfrak{p}).$$ 
\end{defin}

Regarding archimedean places, let us recall what the injectivity diameter is. 
\begin{defin}
Let $A$ be a complex abelian variety. Let $L$ be a polarization on $A$. Let $T_A$ be the tangent space of $A$, let $\Gamma_A$ be its period lattice and $H_L$ the associated Riemann form on $T_A$. The injectivity diameter is the positive number $\displaystyle{\rho(A,L)=\min_{\gamma\in{\Gamma\backslash \{0\}}}\sqrt{H_L(\gamma,\gamma)}}$, \textit{i.e.} the first minimum in the successive minima of the period lattice of $A$.
\end{defin}

\section{A lower bound for the Faltings height}\label{section lower}

We start by recalling Masser's Matrix Lemma in Bost version (later precised by Autissier and Gaudron-R\'emond). We then give a lower bound for the Faltings height by the norm of the bad reduction primes in the semi-stable case, then we obtain the result in the general case by base change, hence deriving a proof of Theorem \ref{height conductor intro} and Corollary \ref{height conductor injectivity}. This implies an upper bound on the Mordell-Weil rank of abelian varieties over number fields in terms of the Faltings height.

\subsection{Archimedean places}

Let us start by the Matrix Lemma given in Th\'eor\`eme 1.1 page 345 of Gaudron and R\'emond \cite{GaRe3} (see also Autissier's \cite{Aut} for good explicit constants if the polarization is principal ; the first version was given by Bost for principally polarized abelian varieties, as stated in Autissier's paper). We give it here with the height $\hFplus(A/K)=h_F(A/K)+\frac{g}{2}\log(2\pi^2)$.

\begin{thm} (Matrix Lemma)\label{matrix}
Let $K$ be a number field such that $A$ is defined over $K$, polarized by an ample line bundle $L$. For any archimedean place $v$ of $K$, denote by $\rho(A_v, L_v)$ the injectivity diameter of the complex polarized abelian variety $(A_v, L_v)$, then
$$\frac{1}{d}\sum_{v\in{M_K^{\infty}}}d_v\rho(A_v,L_v)^{-2} \leq 16\hFplus(A/\overline{\mathbb{Q}})+39g.$$
\end{thm}

The Matrix Lemma is true for the stable height $\hFplus(A/\overline{\mathbb{Q}})$, and we always have $\hFplus(A/K)\geq \hFplus(A/\overline{\mathbb{Q}})$. Here the polarization is not necessarily principal.

\subsection{Bad reduction places}

We compare the height and the size of the bad primes of $A$ over the base field $K$. We first give a proof of the inequality in the semi-stable case and then obtain the general result using base change properties given in the next paragraph. The following proposition gives the result in the semi-stable case. Let us first recall the case of elliptic curves, studied in \cite{Paz14}, where the argument is direct and produces easy constants. Let $A=E$ be an elliptic curve. One has the exact formula 
\[
\hFplus(E/K)=\frac{1}{12d}\left[\log \mathcal{N}_{K/\mathbb{Q}}(\Delta_E)-\sum_{v\in{M_K^{\infty}}}d_v \log\Big(\vert \Delta(\tau_v) \vert (2\Ima\tau_v)^6\Big)\right],
\]
where $\Delta_E$ is the minimal discriminant of the curve, $\tau_v$ is a period in the fundamental domain such that $E(\overline{K}_v)\simeq\mathbb{C}/\mathbb{Z}+\tau_v\mathbb{Z}$ and $\Delta(\tau_v)=q\prod_{n=1}^{+\infty}(1-q^n)^{24}$ is the modular discriminant, with $q=\exp(2\pi i\tau_v)$. A direct analytic estimate using $\Ima\tau_v\geq \sqrt{3}/2$ provides us with 
\begin{equation}\label{elliptic}
\hFplus(E/K)\geq \frac{1}{12d}\log N^0_{E/K}.
\end{equation}

Let's move on to higher dimension.

\begin{prop}\label{semi stable height conductor}
Let $A/K$ be a semi-stable abelian variety of dimension $g$ and defined over a number field $K$ of degree $d$. Then there exists quantities $c_5=c_5(g)>0$ and $c_6=c_6(g)\in{\mathbb{R}}$ such that $$\hFplus(A/K)\geq c_5 \frac{1}{d}\log N^0_{A/K}+c_6.$$ The explicit values $c_5=(12g)^{-12g^{12g^{3g}}}$ and $c_6=-1/c_5$ are valid. If $A$ is the jacobian of a curve of genus $g\geq 1$, then one can even take $c_5=\frac{1}{12}$ and $c_{6}=0$.
\end{prop}

\begin{proof}
The proof is divided into six steps: we start by the case of jacobians in Step 1. Then for general abelian varieties, we reduce to the case of principally polarized abelian varieties in Step 2 by Zarhin's trick.  We make use of several projective heights (theta height, height \`a la Philippon, \ldots) to work on the inequality in Step 3. Then we explain in Step 4 how to find a curve of small height on a principally polarized abelian variety (by a Bertini Theorem) with the extra constraint that it is defined over a finite extension of $K$ with controlled ramification, that will help us reduce the general case to the first case of jacobians. We show that the abelian variety we started with is a quotient of the jacobian of this curve (by classical arguments) in Step 5 and we can finally conclude (via N\'eron-Ogg-Shafarevich) by putting everything together in Step 6. As $A/K$ is semi-stable, its Faltings height is invariant by number field extension, this will be used in the sequel.
\\

\emph{Step 1}. We start by proving the result for jacobians of curves. If $A=J_C$ is the jacobian of a curve $C$, the argument may be presented as follows. By the arithmetic Noether's formula of \cite{MB} Th\'eor\`eme 2.5 page 496 one has for a curve $C$ of genus $g\geq 1$ (with semi-stable jacobian $J_C$) over a number field $K$ of degree $d$,
$$12d\hFplus(J_C/K)=(\omega_C\cdot \omega_C)+\sum_{\substack{\mathfrak{p}\, prime \\ \mathfrak{p}\subset {\mathcal{O}_K}}} \delta_{\mathfrak{p}}(C)\log\mathcal{N}_{K/\mathbb{Q}}(\mathfrak{p})+\sum_{\sigma:K\hookrightarrow \mathbb{C}}\delta(C_\sigma)+dg\log(2^2\pi^8),$$ where the auto-intersection $(\omega_C\cdot \omega_C)$ is non-negative, $\delta(C_\sigma)$ is the delta invariant of Faltings of the complex curve $C_\sigma$ and $\delta_{\mathfrak{p}}(C)$ is the number of singular points in the fiber $C_\mathfrak{p}$. It is zero if and only if $\mathfrak{p}$ is a prime ideal in $\mathcal{O}_K$ of good reduction for $C$. A remark is that the quantity 
\begin{equation}\label{discri}
\frac{1}{d}\sum_{\mathfrak{p}\, prime} \delta_{\mathfrak{p}}(C)\log\mathcal{N}_{K/\mathbb{Q}}(\mathfrak{p})
\end{equation}
is invariant by number field extension of the base field $K$. Indeed, if one proceeds with a base change from $\mathcal{O}_K$ to $\mathcal{O}_{K'}$, each double point in the fiber over a prime $\mathfrak{p}$ of $C/K$ becomes singular in the fiber over primes $\mathfrak{p}'\vert \mathfrak{p}$ of $C/K'$ with thickness equal to the ramification index $e_{\mathfrak{p}'/\mathfrak{p}}$, so the number of double points gets multiplied by $e_{\mathfrak{p}'/\mathfrak{p}}$ by passing from $\mathfrak{p}$ to $\mathfrak{p}'$, see the proof of Lemma 1.12 in \cite{DeMu69}. 

One has $(\omega_C\cdot \omega_C)\geq 0$ and $\delta(C_\sigma)\geq -2g\log2\pi^4$ by \cite{Wil16}, hence $$(\omega_C\cdot \omega_C)+\sum_{\sigma:K\hookrightarrow \mathbb{C}}\delta(C_\sigma)\geq d\cdot c_7(g)$$ where one can take $c_7(g)=-2g\log2\pi^4$. (Note that using the second inequality of Proposition 2.4.8 page 102 of \cite{Java14} we would get $(\omega_C\cdot \omega_C)+\sum_{\sigma:K\hookrightarrow \mathbb{C}}\delta(C_\sigma)\geq -4dg^2$.)

It proves that the height of $J_C$ satisfies
\begin{equation}\label{curves}
\hFplus(J_C/K)\geq \frac{1}{12d}\sum_{\mathfrak{p}\, prime} \delta_{\mathfrak{p}}(C)\log\mathcal{N}_{K/\mathbb{Q}}(\mathfrak{p})
+c_8(g),
\end{equation}

\noindent for $c_8(g)=g\log(2^2\pi^8)-2g\log2\pi^4=0$. This completes the statement for jacobians, because any bad prime for $J_C$ is also a bad prime for $C$, so we have $\delta_{\mathfrak{p}}(C)\geq 1$ for any bad prime of $J_C$. We now aim for a way to reduce to the case of jacobians.
\\

\emph{Step 2}. We may assume, using Zarhin's trick, that the abelian variety is principally polarized. Indeed if $\check{A}$ stands for the dual of $A$, the abelian variety $Z(A)=A^4\times \check{A}^4$ carries a principal polarization, $\hFplus(Z(A)/K)=8\hFplus(A/K)$ and $N^{0}_{Z(A)/K}=N^{0}_{A/K}$. It will have a little cost on the value of the quantities $c_5$ and $c_6$. Let us now fix a principal polarization $L$ on $A$. 
\\

\emph{Step 3}. We will use the theory of Mumford theta coordinates as in the article of \cite{DavPhi} pages 646--652, provided we do a field extension $K'/K$ that enables us to choose a Mumford theta structure of level $4$. The choice $K'=K[A[16]]$ is valid, and Lemma 4.7 page 2078 of \cite{GaRe2} implies 
\begin{equation}\label{deg16}
[K':K]\leq 16^{4g^2}.
\end{equation}

We choose an embedding $\Theta_{16}: A\to \mathbb{P}^{16^g -1}$ given by the theta sections of $L^{\otimes 16}$ and we define the theta height of $(A,L)$ by $h_{\Theta}(A,L)=h(\Theta_{16}(O_A))$. We will in fact show the lower bound for the theta height of $A$: by virtue of the following inequality given in \cite{Paz3} 
\begin{equation}\label{thetafaltings}
\vert h_{\Theta}(A,L)-\frac{1}{2}\hFplus(A/\overline{\mathbb{Q}})\vert \leq 6\cdot4^{2g}\log(4^{2g}) \log( h_{\Theta}(A,L)+2),
\end{equation}
it will lead to the lower bound we seek for the Faltings height of $A$ as explained in Step 6.

By Proposition 3.9 of \cite{DavPhi} page 665, one has for any algebraic subvariety $V\subset A$ the inequality (where $N=16^g-1$) $$\vert \widehat{h}_{\mathbb{P}^N}(V)-h_{\mathbb{P}^N}(V)\vert \leq c_{9}(g, \dim V, \deg V, h_{\Theta}(A,L)),$$ where ${h}_{\mathbb{P}^N}(V)$ is the height of the variety $V$ as defined in \cite{DavPhi} page 644, the height $\widehat{h}_{\mathbb{P}^N}(V)$ is defined in \cite{Phi} in Proposition 9 and the quantity $c_{9}(g, \dim V, \deg V, h_{\Theta}(A,L))>0$ can be taken to be $(4^{g+1}h_{\Theta}(A,L)+3g\log2)\cdot(\dim V+1)\cdot\deg V$. Picking $V=A$, one gets $\widehat{h}_{\mathbb{P}^N}(A)=0$, $\dim A=g$, $\deg_LA = g!$ (the polarization is principal) and 
\begin{equation}\label{h1}
h_{\mathbb{P}^N}(A)\leq c_{10}(g) (h_{\Theta}(A,L)+1),
\end{equation}
 where $c_{10}(g)>0$ only depends on the dimension of $A$, and one can take $c_{10}(g)=4^{g+1}(g!) (g+1)$. Hence giving a lower bound on the height $h_{\mathbb{P}^N}(A)$ will imply a lower bound on the theta height, which in turn will imply a lower bound on the Faltings height by (\ref{thetafaltings}).

By Theorem 1.3 page 760 of \cite{Remond} and Proposition 1.1 page 760 of \cite{Remond} one has that for any curve $C$ in $\mathbb{P}^N$ of genus $g_0$ and degree $\deg C$ there exists a quantity $c_{11}(g_0,\deg C)>0$ such that 
\begin{equation}\label{h2}
h_{\Theta}(J_C,L_\Theta)\leq c_{11}(g_0,\deg C) (h_{\mathbb{P}^N}(C)+1),
\end{equation}
where $L_{\Theta}$ is the polarization associated to the theta divisor on $J_C$. As $C$ is embedded into its jacobian by a theta embedding, one has $\deg(C)=g_0$ and one can even take $c_{11}=(6g_0)^{121g_0 8^{g_0}}$.
\\

\emph{Step 4}. The next goal is now to find an algebraic curve $C$ on $A$ of genus $g_0\leq c_{12}(g)$ such that $h_{\mathbb{P}^N}(C)\leq c_{12}(g) h_{\mathbb{P}^N}(A)$. The proposition is already proved for $g=1$, we may well suppose that $g\geq 2$ from now on. We will cut $A$ by $g-1$ hyperplanes $H_1, ..., H_{g-1}$ in general position of height $h(H_i)\leq c_{13}(g) h_{\mathbb{P}^N}(A)$. 
Using Bertini's Theorem given in Theorem II.8.18 of \cite{Hart} page 179, there exists a dense open subset $U$ such that for any hyperplane $H$ in $U$, the intersection $A\cap H$ is non-singular and connected. As $\overline{\mathbb{Q}}$ is algebraically closed, one has $U(\overline{\mathbb{Q}})\neq \emptyset$, so there exist hyperplanes $H$ with coordinates in $\overline{\mathbb{Q}}$ and $A\cap H$ a geometrically connected non-singular variety in $\mathbb{P}^N$. To be able to choose hyperplans $H_i$ with height $h(H_i)\leq c_{13}(g) h_{\mathbb{P}^N}(A)$, we use the following argument: assume we have an infinite set $S_M$ of algebraic numbers of height less that $M$, where $M\geq 0$ is a fixed real number. This set can be infinite because we don't impose an upper bound the degree of these algebraic numbers. Consider the infinite set of all lines in the dual projective space $\check{\mathbb{P}}^{N}$ with coefficients in $S_M$. As $U$ is an open dense subset, its complement can't contain all these lines, so there exists infinitely many lines intersecting $U$. Pick one of these lines. It provides us with the desired hyperplane $H_i$ in $\mathbb{P}^N$. Repeat the argument $g-1$ times to obtain a smooth curve $C$, geometrically connected on $A$, of genus $g_0$. Furthermore, we would like to ensure that the resulting field extension used to define $C$ will remain as little ramified as possible. The choice of the set $S_M$ is then crucial, we will now take the time to explain how it is done.

Classical existence theorems for infinite unramified extensions of a given number field often come from the application of the Golod-Shafarevich inequality (see \cite{GoSha64}). A quadratic field with at least 5 different prime factors generally admits such an extension. The following result is of a similar spirit. Let $k=\mathbb{Q}(\sqrt{-643\cdot 1318279381})$. By Maire \cite{Mai00}, the quadratic field $k$ admits an infinite everywhere unramified extension $k^\dagger$, which is a tower of unramified 2-extensions. Let $K'k$ be the compositum of $K'$ and $k$ over $\mathbb{Q}$ and let $K_{\infty}''=k^\dagger K'k$ be the compositum of $k^\dagger$ and $K'k$ over $k$. Then $K_{\infty}''/K'k$ is unramified (classical, see Proposition B.2.4 page 592 of \cite{BoGu07} for instance). We want to find small algebraic numbers in this infinite extension.

Let $F\subset k^\dagger$ be a finite extension of $k$. By applying Minkowski's convex body Theorem as in the proof of Theorem B.2.14 page 595 of \cite{BoGu07}\footnote{See also \cite{VaWi13} for better bounds in some cases.}, there exists a non-zero algebraic number $\alpha_F$ in $\mathcal{O}_F$ generating $F$ over $\mathbb{Q}$ (this is important) and with logarithmic absolute Weil height less than $\log\vert\Delta_{F/\mathbb{Q}}\vert^{1/[F:\mathbb{Q}]}$. Now $\vert\Delta_{F/\mathbb{Q}}\vert^{1/[F:\mathbb{Q}]}=\vert\Delta_{k/\mathbb{Q}}\vert^{1/2}$ because $F/k$ is unramified and $[k:\mathbb{Q}]=2$, and $\log\vert\Delta_{k/\mathbb{Q}}\vert^{1/2}$ is bounded from above by $\log10^6<14$. Varying $F$ along the tower, we get infinitely many $\alpha_F$ because each $\alpha_F$ is primitive in $F$, hence they are pairwise distincts. We gather all these $\alpha_F$ to define the set $S_M\subset{K''_\infty}$, for $M=14$, and thus get a curve $C$ defined over a finite\footnote{In the end the curve is defined with a finite number of coefficients in $S_M$ and $K'$.} extension $K''\subset{K_{\infty}''}$ unramified over $K'k$ and with $c_{13}(g)=14$. 

Note that $[K'k:\mathbb{Q}]\leq [k:\mathbb{Q}][K':\mathbb{Q}]=2[K':\mathbb{Q}]$, and $[K'k:\mathbb{Q}]=[K'k:K'][K':\mathbb{Q}]$, hence 
\begin{equation}\label{degree}
[K'k:K']\leq 2.
\end{equation}

Here is a picture to help the reader follow the construction.

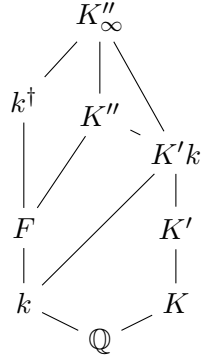
\begin{center}
\begin{tikzpicture}
   
    \node (A) at (0,0.7) {$k^{\dagger}$};
    \node (B) at (1,1.8) {$K''_{\infty}$};
    \node (C) at (0,-2) {$k$};
    \node (D) at (2,-2) {$K$};
    \node (E) at (1,-2.5) {$\mathbb{Q}$};
    \node (F) at (2,-1) {$K'$};
    \node (G) at (2,0) {$K'k$};
    \node (H) at (0,-1) {$F$};
    \node (I) at (1,0.5) {$K''$};
    \draw (C) -- (E);
    \draw (B) -- (G);
    \draw (D) -- (F);
    \draw (D) -- (E);
    \draw (G) -- (C);
    \draw (G) -- (F);
    \draw (B) -- (A);
    \draw (H) -- (A);
    \draw (H) -- (C);
    \draw (I) -- (B);
    \draw (I) -- (H);
    \draw (I) -- (G);
\end{tikzpicture}
\end{center}


The control on the height of the intersection defining $C$ and on the degree of the intersection is provided by Proposition 2.3 page 765 of \cite{Remond} which gives in our situation, as $\deg A=g!$ and $h(H_i)\leq c_{13}(g) h_{\mathbb{P}^N}(A)$,
\begin{equation}\label{h3}
h_{\mathbb{P}^N}(C)\leq h_{\mathbb{P}^N}(A\cap H_1\cap \cdots \cap H_{g-1})\leq c_{14}(g) (h_{\mathbb{P}^N} (A)+1),
\end{equation}

\noindent where one can take $c_{14}(g)=g(g!)+14$.
We also need to control the genus $g_0$ of the curve $C$. Using calculations on the successive Hilbert polynomials of $A\cap H_1\cap...\cap H_i$, one can take the explicit bound $g_0\leq 6^g g! g^2 =c_{12}(g)$, see \cite{CaTa} for the details\footnote{Let us also remark that one can embed the curve in $\mathbb{P}^3$, then using a theorem of Castelnuovo for curves in $\mathbb{P}^3$ given in Theorem 6.4 page 351 of \cite{Hart}, one has $g_0\leq \deg_{\mathbb{P}^3}(C)^2$. }.
\\

\emph{Step 5}. The conjunction of (\ref{h2}), (\ref{h3}), (\ref{h1}) and the fact that $g_0\leq c_{12}(g)$ imply that over the finite extension $K''/K$ there exists a curve $C$ on $A$ such that 
\begin{equation}\label{JC}
h_{\Theta}(J_C,L_\Theta)\leq c_{15}(g) (h_{\Theta}(A,L)+1)
\end{equation}

\noindent where one can take $c_{15}(g)=(12g)^{12g^{12g^{2g}}}$, and by the universal property of the jacobian, from the inclusion of $C$ in $A$ one has a homomorphism $J_C\to A$. Let us show that it is surjective (we follow the classical arguments given for instance in Proposition 6.1 page 104 and Theorem 10.1 page 116-117 of \cite{Mil}). Let $A_1$ be the image of the map, it is an abelian variety. Suppose $A_1\neq A$, we will derive a contradiction. There exists another non-zero abelian subvariety $A_2$ in $A$ such that $A_1\times A_2\to A$ is an isogeny. In particular, $A_1\cap A_2$ is finite, so $C\cap A_2$ is finite because $C$ is in $A_1$, and non-empty because one can always assume $O\in{C}$. Let $W=A_1\times A_2$. It is in particular a non-singular projective variety. Let $\pi$ be the composition of $Id\times [2]:A_1\times A_2\to A_1\times A_2$ with $A_1\times A_2\to A$. Then if $p_2$ denotes the projection on the second factor, $p_2(\pi^{-1}(C))=[2]^{-1}(C\cap A_2)$, so $\pi^{-1}(C)$ is not geometrically connected. But it must be by Corollary 7.9 page 244 of \cite{Hart} (or by lemma 10.3 of \cite{Mil}). This is the desired contradiction, hence $A_1=A$.

This implies that there exists an abelian variety $B$ such that $J_C$ is isogenous to $A\times B$.  Isogenous abelian varieties share the same bad reduction primes by the N\'eron-Ogg-Shafarevich criterion, because they have the same Tate modules (see Theorem 1 page 493 of \cite{SeTa} and Corollary 2 page 22 of \cite{Falt2}). Thus, if we denote $d''=[K'':\mathbb{Q}]$, we get that 
\begin{equation}\label{JCA}
 \frac{1}{d''}\sum_{\mathfrak{p}''\, \mathrm{bad}\, \mathrm{for}\, A} \delta_{\mathfrak{p}''}(C)\log\mathcal{N}_{K''/\mathbb{Q}}(\mathfrak{p}'') \leq \frac{1}{d''}\sum_{\mathfrak{p}''\,\subset\,O_{K''}} \delta_{\mathfrak{p}''}(C)\log\mathcal{N}_{K''/\mathbb{Q}}(\mathfrak{p}'').
\end{equation}

\emph{Step 6}. Let us show that we have reduced the proof to the case of jacobians of curves. Following the previous steps we get
$$ \frac{1}{d''}\sum_{\mathfrak{p}''\,\subset\,O_{K''}} \delta_{\mathfrak{p}''}(C)\log\mathcal{N}_{K''/\mathbb{Q}}(\mathfrak{p})\underset{(i)}{\leq} \hFplus(J_C/K'')\underset{(ii)}{\ll} h_{\Theta}(J_C,L_\Theta)\underset{(iii)}{\ll} h_{\Theta}(A,L)\underset{(iv)}{\ll} \hFplus(A/K''),$$ where the implied constants depend only on $g$ and the successive inequalities are


\begin{itemize}
\item (i) is the case of curves given by inequality (\ref{curves}),
\item (ii) is the comparison between the theta height and the Faltings height of \cite{Paz3} as recalled in (\ref{thetafaltings}),
\item (iii) is inequality (\ref{JC}),
\item (iv) is again (\ref{thetafaltings}).
\end{itemize}

\noindent If the curve $C$ was defined over $K$, we could use on the far left side the invariance property (\ref{discri}). In the general case we have nevertheless inequality (\ref{JCA}), and we get from there
 \begin{equation}\label{ramifiii}
\frac{1}{d''}\sum_{\substack{\mathfrak{p}''\subset\mathcal{O}_{K''}\\ \mathrm{bad}\, \mathrm{for}\, A}} \delta_{\mathfrak{p}''}(C)\log\mathcal{N}_{K''/\mathbb{Q}}(\mathfrak{p}'')\geq  \frac{1}{d}\sum_{\substack{\mathfrak{p}\subset{\mathcal{O}_{K}}\\ \mathrm{bad}\, \mathrm{for}\, A}} \Big(\sum_{\mathfrak{p}''\vert \mathfrak{p}} \frac{ f_{\mathfrak{p}''/\mathfrak{p}}}{[K'':K]}\Big)\log\mathcal{N}_{K/\mathbb{Q}}(\mathfrak{p})
 \end{equation}
\noindent where $f_{\mathfrak{p}''/\mathfrak{p}}$ is the residual degree. Indeed, if $\mathfrak{p}''$ is a bad prime for $A$, it is also a bad prime for $\Jac(C)$, hence also a bad prime for $C$ (the converse statement is wrong), hence $\delta_{\mathfrak{p}''}(C)\geq 1$. Using $$[K'':K]=\sum_{\mathfrak{p}''\vert \mathfrak{p}}e_{\mathfrak{p}''/\mathfrak{p}}f_{\mathfrak{p}''/\mathfrak{p}}\leq \Big(\max_{\mathfrak{p}''\vert \mathfrak{p}} e_{\mathfrak{p}''/\mathfrak{p}}\Big)\sum_{\mathfrak{p}''\vert \mathfrak{p}}f_{\mathfrak{p}''/\mathfrak{p}},$$ one gets in (\ref{ramifiii})
$$ \frac{1}{d''}\sum_{\mathfrak{p}''\, \mathrm{bad}\, \mathrm{for}\, A} \delta_{\mathfrak{p}''}(C)\log\mathcal{N}_{K''/\mathbb{Q}}(\mathfrak{p}'')\geq \frac{1}{d}\sum_{\mathfrak{p}\, \mathrm{bad}\, \mathrm{for}\, A} \frac{1}{\displaystyle{\max_{\mathfrak{p}''\vert \mathfrak{p}}} e_{\mathfrak{p}''/\mathfrak{p}}}\log\mathcal{N}_{K/\mathbb{Q}}(\mathfrak{p})\geq \frac{1}{2\cdot16^{4g^2} d}\log N^0_{A/K}$$
where the last inequality holds because the ramification index is controlled by $e_{\mathfrak{p}''/\mathfrak{p}}\leq [K'k:K'][K':K]\leq 2\cdot 16^{4g^2}$ as $K''/K'k$ is unramified and as one has (\ref{deg16}) and (\ref{degree}).

One concludes by $\hFplus(A/K'')=\hFplus(A/K)$ on the far right side because $A/K$ is already semi-stable. At each and every step an explicit constant is provided, an easy calculation leads to $c_5=(12g)^{-12g^{12g^{3g}}}$ and $c_6=-1/c_5$ for the general case. These values are not expected to be optimal. 
\end{proof}

\subsection{Reducing to the semi-stable case}

We explain in this section how to use base change properties to derive the general case from the semi-stable case. Let us start by the following definition.

\begin{defin}
Let $A$ be an abelian variety defined over a discrete valuation field $K_\mathfrak{p}$ and let $K_{\mathfrak{p}'}'$ be a finite extension of $K_\mathfrak{p}$ where $A$ has semi-stable reduction, with ramification index $e_{\mathfrak{p}'/\mathfrak{p}}$, where $\mathfrak{p}'$ is a prime above $\mathfrak{p}$, and $\omega_{A/K_{\mathfrak{p}}}$ the determinant of differentials. Let $h_{\mathfrak{p}}: \mathcal{A}_{\mathcal{O}_{K_\mathfrak{p}}}\times_{\mathcal{O}_{K_{\mathfrak{p}}}}\mathcal{O}_{K'_{\mathfrak{p}'}}\to \mathcal{A}_{\mathcal{O}_{K'_{\mathfrak{p}'}}}$ be the canonical base change morphism. Let $\mathrm{Lie}(h_\mathfrak{p})$ be the induced injective morphism on differentials.
Let $$c(A,\mathfrak{p},\mathfrak{p}')=\frac{1}{e_{\mathfrak{p}'/\mathfrak{p}}}\mathrm{length}_{\mathcal{O}_{K_{\mathfrak{p}'}'}}\mathrm{coker(Lie(h_{\mathfrak{p}}))}$$
be the base change conductor, where if $\Gamma(.,.)$ stands for global sections one has
$$\mathrm{coker(Lie(h_{\mathfrak{p}}))}=\frac{\Gamma(\Spec(\mathcal{O}_{K_{\mathfrak{p}}}),\omega_{A/K_{\mathfrak{p}}})\otimes \mathcal{O}_{K_{\mathfrak{p}'}'}}{\Gamma(\Spec(\mathcal{O}_{K_{\mathfrak{p}'}'}), \omega_{A/K_{\mathfrak{p}'}'})}.$$ 
\end{defin}

This conductor was defined by Chai in \cite{Cha}, see also the reference \cite{HaNi} pages 90-98. It verifies in particular the two following key properties.
\begin{prop}\label{c=0}
Let $A$ be an abelian variety defined over a discrete valuation field $K_\mathfrak{p}$ and let $K_{\mathfrak{p}'}'$ be a finite extension of $K_\mathfrak{p}$ where $A$ has semi-stable reduction with base change conductor $c(A,\mathfrak{p},\mathfrak{p}')$. Then one has
\begin{enumerate}
\item $c(A,\mathfrak{p},\mathfrak{p}')=0$ if and only if $A/K_{\mathfrak{p}}$ has semi-stable reduction,
\item if $A$ is not semi-stable at $\mathfrak{p}$, then $c(A,\mathfrak{p},\mathfrak{p}')\geq 1/e_{\mathfrak{p}'/\mathfrak{p}}$.
\end{enumerate}
\end{prop}

\begin{proof}
The proof goes along the same lines as Proposition 4.3 of \cite{Paz14} which deals with elliptic curves. As it is relatively short, we give it here for abelian varieties. Let us start by assuming that $A/K_{\mathfrak{p}}$ has semi-stable reduction. Denote by $\mathcal{A}^0_{\mathcal{O}_{K_\mathfrak{p}}}$ the identity component of the N\'eron model of $A$ over $K_{\mathfrak{p}}$, one then has $\mathcal{A}^0_{\mathcal{O}_{K_\mathfrak{p}}}\otimes\mathcal{O}_{K_{\mathfrak{p}'}'}\simeq \mathcal{A}^0_{\mathcal{O}_{K_{\mathfrak{p}'}'}}$ by Corollaire 3.3 page 348 of SGA 7.1 \cite{SGA}, hence the differentials are the same, so $c(A,\mathfrak{p},\mathfrak{p}')=0$.

Reciprocally, one still has a map $\Phi: \mathcal{A}^0_{\mathcal{O}_{K_\mathfrak{p}}}\otimes\mathcal{O}_{K_{\mathfrak{p}'}'} \to\mathcal{A}^0_{\mathcal{O}_{K_{\mathfrak{p}'}'}}$. As $c(A,\mathfrak{p},\mathfrak{p}')=0$, the Lie algebras are the same and as $\Phi$ is an isomorphism on the generic fibers, $\Phi$ is birational. On the special fiber, $\Phi$ has finite kernel and is thus surjective because the dimensions are equal, here again because $c(A,\mathfrak{p},\mathfrak{p}')=0$.

We have that $\Phi$ is quasi-finite and birational. As $\mathcal{A}_{\mathcal{O}_{K_{\mathfrak{p}'}'}}$ is normal, by Zariski's Main Theorem found in Corollary 4.6 page 152 of \cite{Liu} for instance, $\Phi$ is an open immersion. So $\Phi$ is surjective and is also an open immersion, hence an isomorphism. This implies that $A/K_{\mathfrak{p}}$ is semi-abelian, and proves part (1). Part (2) is easier, if $A$ is not semi-stable then the length in the definition of $c(A,\mathfrak{p},\mathfrak{p}')$ is a positive integer, hence bigger than $1$.
\end{proof}

We need a lemma.
\begin{lem}
Let $Uns$ denote the set of unstable primes of $A$ over $K$. Let $K'$ be a number field extension of $K$ over which $A$ has semi-stable reduction everywhere. Then one has
\begin{equation}\label{base change conductor}
\hFplus(A/K)-\hFplus(A/K') \geq \frac{1}{[K':\mathbb{Q}]}\sum_{\mathfrak{p}\in{Uns}}\log \mathcal{N}_{K/\mathbb{Q}}(\mathfrak{p}). 
\end{equation}
\end{lem}
\begin{proof}
For a field $F$, we denote by $\mathcal{A}_{\mathcal{O}_F}$ the N\'eron model of $A$ over the base $\Spec{{\mathcal{O}_F}}$.
As $K'$ is a finite extension of $K$, we have 
\begin{equation}\label{chaichai}
\hFplus(A/K)-\hFplus(A/K')=\frac{1}{[K:\mathbb{Q}]}\deg(\omega_{\mathcal{A}_{\mathcal{O}_K}})-\frac{1}{[K':\mathbb{Q}]}\deg({\omega_{\mathcal{A}_{\mathcal{O}_{K'}}}}),
\end{equation}
\textit{i.e.} the archimedean parts cancel out.
Let $\phi:K\to K'$ be the inclusion, we have a morphism $\omega_{\mathcal{A}_{\mathcal{O}_{K'}}}\to \phi^* \omega_{\mathcal{A}_{\mathcal{O}_{K}}}$, taking degrees (see also the proof of Lemme 1.23 page 35 of \cite{Del}) one obtains $$[K':K]\deg(\omega_{\mathcal{A}_{\mathcal{O}_{K}}})=\deg(\phi^*\omega_{\mathcal{A}_{\mathcal{O}_{K}}})=\deg(\omega_{\mathcal{A}_{\mathcal{O}_{K'}}})+\sum_{\mathfrak{p}\subset\mathcal{O}_K}\sum_{\mathfrak{p}'\vert\mathfrak{p}}\mathrm{length}_{\mathcal{O}_{K'_{\mathfrak{p}'}}}(\mathrm{coker\phi}) \log\mathcal{N}_{K'/\mathbb{Q}}(\mathfrak{p}'),$$
hence using Proposition \ref{c=0} we obtain
$$[K':K]\deg(\omega_{\mathcal{A}_{\mathcal{O}_{K}}})\geq\deg(\omega_{\mathcal{A}_{\mathcal{O}_{K'}}})+\sum_{\mathfrak{p}\in{Uns}}\sum_{\mathfrak{p}'\vert\mathfrak{p}}\frac{[K':K]}{e_{\mathfrak{p}'/\mathfrak{p}}} \log\mathcal{N}_{K/\mathbb{Q}}(\mathfrak{p}),$$
hence dividing by $[K':\mathbb{Q}]=[K':K][K:\mathbb{Q}]$, and using $e_{\mathfrak{p}'/\mathfrak{p}}\leq [K':K]$
$$\frac{1}{[K:\mathbb{Q}]}\deg(\omega_{\mathcal{A}_{\mathcal{O}_{K}}})\geq\frac{1}{[K':\mathbb{Q}]}\deg(\omega_{\mathcal{A}_{\mathcal{O}_{K'}}})+\frac{1}{[K':\mathbb{Q}]}\sum_{\mathfrak{p}\in{Uns}} \log\mathcal{N}_{K/\mathbb{Q}}(\mathfrak{p}),$$
which gives the result by using (\ref{chaichai}).
\end{proof}

We are now ready to perform a base change on the inequality of Proposition \ref{semi stable height conductor}, which will be the completion of the proof of Theorem \ref{height conductor intro}. 

\begin{prop}\label{height conductor}(Final step in the proof of Theorem \ref{height conductor intro}.)
Let $g\geq 1$ be an integer and $K$ a number field of degree $d$. There exists $c_{16}(g)>0$ and $c_{17}(g)\in{\mathbb{R}}$ such that for any abelian variety (not necessarily semi-stable) $A$ defined over $K$, with dimension $g$, one has $$\hFplus(A/K) \geq c_{16} \frac{1}{d}\log N^{0}_{A/K}+c_{17},$$ and one can take $c_{16}=c_5/12^{4g^2}$ and $c_{17}=c_6$. If $A$ is the jacobian of a curve, one can take $c_{16}=1/12^{4g^2+1}$ and $c_{17}=0$.
\end{prop}

\begin{proof}
 Let $N^{st}_{A/K}$ be the product of the norms of primes where $A$ has semi-stable bad reduction. Let $N^{uns}_{A/K}$ be the product of the norms of primes where $A$ has unstable bad reduction. By definition one has $N^0_{A/K}=N^{st}_{A/K}N^{uns}_{A/K}$. Let $K'$ be a number field extension of $K$ such that $A$ acquires semi-stable reduction everywhere over $K'$. Using equality (\ref{base change conductor}), one gets $$\hFplus(A/K)\geq \hFplus(A/K') +\frac{1}{[K':\mathbb{Q}]}\log N^{uns}_{A/K}.$$ As $A/K'$ has semi-stable reduction everywhere, one obtains by Proposition \ref{semi stable height conductor} that $$\hFplus(A/K')\geq c_{5}(g)\frac{1}{d'} \log N^{st}_{A/K'}+c_{6}(g).$$ Recall (use Theorem 6.2 page 413 of \cite{SiZa95}) that one may choose the explicit extension $K'=K[A[12]]$, hence the degree $d'=[K':\mathbb{Q}]$ is controlled by the degree $d=[K:\mathbb{Q}]$ and by the dimension of $A$; for instance, apply Lemma 4.7 page 2078 of \cite{GaRe2} to obtain $d'=[K[A[12]]: K]\leq 12^{4g^2}$. Now one has $$\frac{1}{d'} \log N^{st}_{A/K'}=\frac{1}{d'}\sum_{\mathfrak{p}'\subset\mathcal{O}_{K'}}\log\mathcal{N}_{K'/\mathbb{Q}}(\mathfrak{p}')\geq\frac{1}{d}\sum_{\mathfrak{p}\subset\mathcal{O}_K}\frac{1}{\displaystyle{\max_{\mathfrak{p}'\vert\mathfrak{p}}e_{\mathfrak{p}'/\mathfrak{p}}}}\log\mathcal{N}_{K/\mathbb{Q}}(\mathfrak{p})\geq\frac{1}{12^{4g^2}d}\log N^{st}_{A/K}$$ 
 because $e_{\mathfrak{p}'/\mathfrak{p}}\leq [K':K]$ and so gathering the estimates we obtain $$\hFplus(A/K)\geq c_{18}\frac{1}{d}\log N^{st}_{A/K}+ c_{19}\frac{1}{d}\log N^{uns}_{A/K}+c_{17}\geq \min\{c_{18}, c_{19}\}\frac{1}{d}\log N^{0}_{A/K}+c_{17},$$ where the quantities $c_{17}$, $c_{18}$, $c_{19}$ only depend on $g$.
\end{proof}

Note that for $g=1$, Proposition \ref{height conductor} is an improvement on Proposition 4.4 page 57 of \cite{Paz14}, both in the result and in the presentation: an equality of prime norms is incorrect in \textit{loc. cit.} because of possible ramification of stable primes of $K'/K$, but the proof fortunately led to a weaker inequality in the end, so the result stated in \textit{loc. cit.} still holds, and anyhow the new result given here is better.

We obtain an easy proof of Corollary \ref{height conductor injectivity} as the sum of Theorem \ref{matrix} and Proposition \ref{height conductor}. Apply Proposition \ref{height conductor} to get $\hFplus(A/K) \geq c_{16} \log N^{0}_{A/K}+c_{17}$ and Theorem \ref{matrix} to get $$16 \hFplus(A/K)\geq 16\hFplus(A/\overline{\mathbb{Q}})\geq \frac{1}{d}\sum_{v\in{M_K^{\infty}}}d_v\rho(A_v,L_v)^{-2}  - 39g,$$
then sum these two inequalities. 

We can now derive Corollary \ref{rank height intro}.
\begin{proof} \textit{(of Corollary \ref{rank height intro})} 
We will use as a pivot the quantity $N^{0}_{A/K}$. Applying Theorem 5.1 of \cite{Remond} page 775, there exists quantities $c_{20}=c_{20}(K,g)>0$ and $c_{21}=c_{21}(K,g)\geq0$ such that $m_{K}\leq c_{20}(K,g) \log N^0_{A/K} + c_{21}(K,g)$. The quantities depend on the degree and the discriminant of the base field here. This last inequality doesn't require semi-stability of $A$.
Applying Proposition \ref{height conductor} of the present text one obtains $\log N^0_{A/K}\leq c_{22}(K,g) \max\{\hFplus(A/K),\,1\}$, also valid in general. Use the explicit quantities (valid in general) of Theorem 5.1 of \cite{Remond} page 775, it leads to $m_{K}\leq 4g^3d^22^{8g^2} \log N^0_{A/K} + gd2^{8g^2}(\log\vert \Delta_K\vert + g^2d^2\log16)$, and combine with Proposition \ref{height conductor}. It proves the corollary. In the case of jacobians combine with Proposition \ref{semi stable height conductor} which gives $\hFplus(J_C/K)\geq \frac{1}{12d}\log N^0_{J_C/K}$ in the semi-stable case and Proposition \ref{height conductor} which gives $\hFplus(J_C/K)\geq \frac{1}{12d12^{4g^2}}\log N^0_{J_C/K}$ for the general case.
\end{proof}

\section{Lang-Silverman conjecture and regulators}\label{section LS}

We give here a conjecture of Lang and Silverman (\cite{Sil3} page 396\footnote{This first version of the conjecture is known to be wrong, consider for instance the point $(P,0)$ on a variety $A_1\times A_2$ where $P$ is non-torsion and let the height of $A_2$ tend to infinity. However the philosophy of the conjecture is clearly the same as the original statement, a generic point can't have too small height.} or \cite{Paz3}\footnote{This stronger version is also known to be wrong, see Remark \ref{nPP} and section \ref{section LS strong} of the present article for a clarification.}). Throughout this section, we will use the notation $\overline{\End(A)\cdot P}=A$ to say that the set $\End(A)\cdot P$ is Zariski dense in $A$.

\begin{conj}({Lang-Silverman})\label{LangSilV1}
Let $g\geq 1$ be an integer. For any number field $K$, there exists a positive quantity $c_{23}=c_{23}(K,g)$ such that for any abelian variety $A/K$ of dimension $g$ and any ample symmetric line bundle $L$ on $A$, for any point $P\in{A(K)}$, one has:
\[
\Big(\overline{\mathrm{End}(A)\!\cdot\! P}=A \Big)\Rightarrow \;\Big(\widehat{h}_{A,L}(P) \geq c_{23}\, \max\Big\{\hFplus(A/K),\,1\Big\}\Big)\;,
\]
where $\widehat{h}_{A,L}(.)$ is the {N\'eron-Tate} height associated to the line bundle $L$ and $\hFplus(A/K)$ is the (relative) Faltings height of the abelian variety $A/K$.
\end{conj}

\begin{rem}\label{nPP}
We only require the condition $\mathrm{End}(A)\cdot P$ Zariski dense, not necessarily $\mathbb{Z}\cdot P$ Zariski dense.
Let us consider the following situation: let $A_1$ be a simple abelian variety over $K$ and let $A=A_1\times A_1$. Choose $P=([n]P_1, P_1)\in{A(K)}$. If $P_1$ is non-torsion, then $\overline{\mathbb{Z}\cdot P}$ is a strict abelian subvariety (of degree growing with $n$), whereas $\overline{\mathrm{End}(A)\cdot P}=A$. As one has $\widehat{h}_{A,L}(P)=(n^2+1)\widehat{h}_{A_1,L_1}(P_1)$ for the product polarization $L$, and as $\hFplus(A/K)=2\hFplus(A_1/K)$, the point $P$ verifies the expected lower bound if the point $P_1$ does.
\end{rem}

\begin{prop}\label{reg height thm}
Assume the Lang-Silverman Conjecture \ref{LangSilV1}. Let $K$ be a number field and $g,m \geq 1$ be integers. There exists a quantity $c_{24}=c_{24}(K,g,m)>0$ such that for any simple abelian variety $A$ defined over $K$ of dimension $g$, of rank $m$ over $K$, polarized by an ample and symmetric line bundle $L$, 
$$ \mathrm{Reg}_{L}(A/K)\geq \Big(c_{24} \max\{\hFplus(A/K),\,1\}\Big)^{m}.$$
\end{prop}

\begin{proof}
Let us denote $h=\max\{\hFplus(A/K),\,1\}$ and for any $i\in\{1, ..., m\}$, the Minkowski $i$th-minimum $\lambda_i=\lambda_i(A(K)/A(K)_{\mathrm{tors}})$. Apply Minkowski's successive minima inequality to the Mordell-Weil lattice, 
\[
 \lambda_1 \cdots \lambda_{m}\leq m^{{m/2}}  (\mathrm{Reg}_{L}(A/K))^{1/2}.
\]

\noindent Now, as $A$ is simple, any non-torsion point verifies $\overline{\mathrm{End}(A)\cdot P}=A$, so using $m$ times the inequality of Conjecture \ref{LangSilV1} one gets
\begin{equation}\label{reg height part}
\mathrm{Reg}_{L}(A/K)\geq \frac{ c_{23}^{m}h^{m}}{m^{{m}}},
\end{equation}
\noindent which gives the result.
\end{proof}

We thus obtain Theorem \ref{reg height thm intro} as a corollary of Proposition \ref{reg height thm}. Indeed if the rank is non zero, as soon as the regulator, the rank and the dimension are bounded from above, the height will be bounded from above, hence the claimed finiteness. This may be expressed in other words by: the Lang-Silverman conjecture implies that the regulator $\Reg_L(A/K)$ verifies a Northcott property on the set of polarized simple abelian varieties (modulo isomorphisms) of dimension $g$ defined over a fixed number field $K$ with $A(K)$ Zariski dense and Mordell-Weil rank bounded from above. 

\begin{rem}
Back to inequality (\ref{reg height part}), in view of Corollary \ref{rank height intro}, we have $h\gg m$. Any improvement of the form $h\gg m^{1+\varepsilon}$ (for a fixed $\varepsilon>0$) would lead to a stronger Northcott property, without assuming that the rank is bounded from above. See also the addendum \cite{Paz16b}.
\end{rem}

\section{A stronger lower bound conjecture}\label{section LS strong}

We would like to refine the conjecture\footnote{Such an attempt has been proposed in Conjecture 1.8 of \cite{Paz3}, but it unfortunately fails because of situations similar to the one described in Remark \ref{nPP} where certain points fall in the first case but should fall in the second instead. This was communicated to the author by the referee of another project, may he be warmly thanked here. We fix the problem by changing the condition given there as $\overline{\mathbb{Z}\cdot P}=A$ by the weaker $\overline{\End(A)\cdot P}=A$. We also add a dependance in $\deg_{L}(A)$ in the attempt to control the degree of $B$ thanks to a remark of Ga\"el R\'emond.} of Lang and Silverman to take care of the exceptional points in Conjecture \ref{LangSilV1}: what can be said about the points $P$ verifying $\overline{\End(A)\cdot P}\subsetneq A$?

\begin{conj}({Lang-Silverman}, new strong version)\label{LangSilV2}
Let $g\geq 1$ be an integer. For any number field $K$, there exists two positive quantities $c_{33}=c_{33}(K,g)$ and $c_{34}=c_{34}(K,g)$ such that for any abelian variety $A/K$ of dimension $g$ and any ample symmetric line bundle $L$ on $A$, for any point $P\in{A(K)}$, one has:
\begin{itemize}
\item[$\bullet$] either there exists an abelian subvariety $B\subset A$, $B\neq A$, of degree $\deg_{L}(B)\leq c_{34}\deg_{L}(A)$ and such that the order of the point $P$ modulo $B$ is bounded from above by $c_{34}$,
\item[$\bullet$] or one has $\mathrm{End}(A)\!\cdot\! P$ is Zariski dense and
\[
\widehat{h}_{A,L}(P) \geq c_{33}\, \max\Big\{\hFplus(A/K),\,1\Big\}\;,
\]
where $\widehat{h}_{A,L}(.)$ is the { N\'eron-Tate} height associated to the line bundle $L$ and $\hFplus(A/K)$ is the (relative) Faltings height of the abelian variety $A/K$.
\end{itemize}
\end{conj}

This is a strong statement. It implies the strong torsion conjecture for example. Indeed, any torsion point $P\in{A(K)_{\mathrm{tors}}}$ falls into the first case because its canonical height is zero. Hence the order of $P$ is bounded from above solely in terms of $K$ and $g$ and of the cardinality of the torsion subgroup of a strict abelian subvariety $B$. An easy induction on the dimension of $A$ gives a bound on the order of $P$ solely in terms of $K$ and $g$, hence on the exponent of the torsion group as well, hence on the cardinal of the torsion group $A(K)_{\mathrm{tors}}$ as well. 

This strong form of the conjecture is motivated by Th\'eor\`eme 1.4 page 511 of \cite{Sinn} and Th\'eor\`emes 1.8 and 1.13 of \cite{Paz13}. Remark that in both of these works, the abelian varieties considered are principally polarized, hence the dependance in the degree of $A$ is only through the dimension $g$.

Let us see now how this statement can help in understanding the link between the Mordell-Weil group $A(K)$ and the abelian subvarieties of $A$. The following quantity will play a key role in this paragraph. 

\begin{defin}
Let $A$ be an abelian variety over a number field $K$. Let $m_K$ denote the Mordell-Weil rank of $A(K)$.
Define $$m_0=\sup\{\mathrm{rank}(B(K))\,\vert  \, B \,\mathrm{strict \,abelian \,subvariety \,of} \,A\}.$$
We will call the relative quantity $m_K-m_0$ the \emph{Zariski rank} of the Mordell-Weil group $A(K)$. 
\end{defin}

Note that $m_K-m_0>0$ is equivalent to $A(K)$ being Zariski dense in $A$. This Zariski rank could be compared with the following quantity for a number field $K$. If $r_K$ is the rank of units in $K$, let $r_0$ denote the maximal rank of units in a strict subfield of $K$. As already noticed in \cite{Paz14} in the easier case of elliptic curves, the Zariski rank $m_K-m_0$ plays the same role (at least when one gives lower bounds on the regulator in both contexts) as the relative rank of units $r_K-r_0$ for number fields.

The next lemma studies the size of the successive minima of the Mordell-Weil lattice modulo torsion, where the square of the norm is implicitly given by the N\'eron-Tate height. We believe this version could lead in the future to some improvements in Theorem \ref{reg height thm intro}.

\begin{lem}\label{success}
Assume Conjecture \ref{LangSilV2}. Let $(A,L)$ be a polarized abelian variety of dimension $g$ defined over a number field $K$. For any $i\in\{1, ..., m_K\},$ let $\lambda_i$ be the $i$-th successive minima of the lattice $A(K)/A(K)_{tors}$. Then there is a quantity $c_{35}=c_{35}(K,g,\deg_{L}(A))>0$ such that 

\begin{center}
$\left\{\begin{tabular}{ll}
for any $i$, & $\lambda_i^2 \geq c_{35}\, i $,\\
\\
if $i> m_0$, & $\lambda_i^2 \geq c_{35} \max\{1, \hFplus(A/K)\}$. 
\end{tabular}\right.$
\end{center}
\end{lem}

\begin{proof}
Within the proof, we will use the symbol $c_{*}$ for a positive quantity only depending on $g$, on $K$ and on $\deg_{L}(A)$. We allow the value of this quantity $c_{*}$ to vary at some steps within the proof, as long as it depends only on $g$, on $K$ and on $\deg_{L}(A)$ and stays positive. If $c_{34}(K,g)$ denotes the quantity appearing in Conjecture \ref{LangSilV2}, denote by $\displaystyle{\overline{c_{34}}=\max\{1,\max_{1\leq i \leq g} c_{34}(K,i)}\}$, the field $K$ being fixed. 

Let $\mathcal{B}$ denote the set of all abelian subvarieties $B$ in $A$ of degree bounded from above by $\overline{c_{34}}^g\deg_L(A)$: it contains the subvarieties appearing in the first case of Conjecture \ref{LangSilV2}, and we raise $\overline{c_{34}}$ to the power $g$ to be able to use an induction on the dimension $g$ towards the end. This is a finite set with cardinal bounded from above by a quantity depending only on $g$, on $K$ (because $\overline{c_{34}}$ only depends on $g$ and $K$) and on $\deg_{L}(A)$. The reader interested in an explicit upper bound for the cardinal of this set can refer to Proposition 4.1 page 529 of \cite{Remond3}.

Choose an integer $i\in\{1, ..., m_K\}$ and define $$Z_i=\bigcup_{\substack{B\in{\mathcal{B}}\\ \mathrm{rank}(B(K))<i}} B(K).$$
The set $A(K)\backslash Z_i$ is non-empty, because the rank of the lattices in the finite union is always strictly smaller than $m_K$. Let $P_i$ be a non-torsion point of minimal height lying in $A(K)\backslash Z_i$. Apply the ``Lemme d'\'evitement'' of Gaudron-R\'emond given in Th\'eor\`eme 1.1 page 125 of \cite{GaRe} to obtain $\widehat{h}_{A,L}(P_i)\leq M \lambda_i(A(K))^2$, where $M$ is positive and bounds from above the cardinality of $\mathcal{B}$, which can be chosen depending only on $g$, on $K$ and on $\deg_L(A)$ in view of Conjecture \ref{LangSilV2}.

On the one hand if $\End(A)\cdot{P_i}$ is dense in $A$, by Conjecture \ref{LangSilV2} one has $\widehat{h}_{A,L}(P_i)\geq c_{*} \max\{1,\hFplus(A/K)\}$. We add that $\max\{1,\hFplus(A/K)\}\geq c_{*} m_K$ by applying the Corollary $\ref{rank height intro}$, and $m_K\geq i$ by definition. 

On the other hand if $\End(A)\cdot{P_i}$ is not dense in $A$, by Conjecture \ref{LangSilV2} there exists an abelian variety $B\in{\mathcal{B}}$ such that the order of $P_i$ is less than $c_{34}$ modulo $B$, then by definition of $m_0$ one has $\mathrm{rank}(B(K))\leq m_0$ and by choice of $P_i$ one has $\mathrm{rank}(B(K))\geq i$, hence $m_0\geq i$. 

Apply Conjecture \ref{LangSilV2} to $P_i\in{B(K)}$ (modulo torsion in $A$). If $\End(B)\cdot P_i$ is dense in $B$ one obtains $$\widehat{h}_{A,L}(P_i)=\widehat{h}_{B,L}(P_i)\geq c_{*}\max\{1, \hFplus(B/K)\},$$ then again using Corollary $\ref{rank height intro}$ one gets $\widehat{h}_{A,L}(P_i)\geq c_{*}\mathrm{rank}(B(K))\geq c_{*} i$. If $\End(B)\cdot P_i$ is not dense in $B$, we are reduced to the case of a strict abelian subvariety of $A$. There exists an abelian subvariety $B_1\subset B$ of degree bounded from above by $c_{34}\deg_L(B)$ such that $P_i$ has order bounded from above by $\overline{c_{34}}$ modulo $B_1$. As $\deg_L(B)\leq \overline{c_{34}}\deg_L(A)$, one has $\deg_L(B_1)\leq \overline{c_{34}}^2 \deg_L(A)$, hence $\deg_L(B_1)\leq \overline{c_{34}}^g \deg_L(A)$ so $B_1\in{\mathcal{B}}$. As $P_i$ avoids $Z_i$ one has again $\mathrm{rank}(B_1(K))\geq i$. If $\End(B_1)\cdot P_i$ is dense in $B_1$, then $$\widehat{h}_{A,L}(P_i)\geq c_{*}\max\{1, \hFplus(B_1/K)\}\geq c_{*} \mathrm{rank}(B_1(K))\geq c_{*} i.$$ If $\End(B_1)\cdot P_i$ is not dense in $B_1$, one continues by induction until one reaches a strict abelian subvariety $B_n$ such than $\End(B_n)\cdot P_i$ is dense in $B_n$, which will eventually be the case when $B_n$ is simple for instance. It gives the lemma.
\end{proof}

\begin{prop}\label{reg height thm V2}
Assume Conjecture \ref{LangSilV2}. Let $K$ be a number field, let $g\geq 1$ be an integer, let $m\geq 0$ be an integer. There exists a quantity $c_{36}=c_{36}(K,g,m)>0$ such that for any principally polarized abelian variety $A$ defined over $K$ of dimension $g$, equipped with an ample and symmetric line bundle $L$, with $A(K)$ of rank $m$,
$$ \mathrm{Reg}_{L}(A/K)\geq \Big(c_{36} \max\{\hFplus(A/K),\,1\}\Big)^{m-m_0}.$$
\end{prop}

\begin{proof}
Let us denote $h=\max\{\hFplus(A/K),\,1\}$ and $m=\mathrm{rank}(A(K))$, and for any $i\in\{1, ..., m\}$, $\lambda_i=\lambda_i(A(K)/A(K)_{\mathrm{tors}})$. 

The inequality is trivial for $m=0$. From now on, let us assume $m\neq 0$. Apply Minkowski's successive minima inequality to the Mordell-Weil lattice, 
\[
 \lambda_1^2 \cdots \lambda_{m}^2 \leq {m}^{m}  \mathrm{Reg}_{L}(A/K).
\]

\noindent Now apply lemma \ref{success} with $\deg_L(A)=g!$ to get
\begin{equation}\label{reg height part2}
\mathrm{Reg}_{L}(A/K)\geq \frac{ c_{35}^{m-m_0}h^{m-m_0}\lambda_1^2 \cdots \lambda_{m_0}^2}{m^{m}},
\end{equation}

%
%
%
%
If $m_0=0$, the inequality is the one claimed. Let us suppose that $m_0\neq 0$. Apply again Lemma \ref{success} to get
\begin{equation}\label{reg m02}
 \mathrm{Reg}_{L}(A/K)\geq m^{-m} (c_{35})^{m_0}\,(m_0!) \left(c_{35} h\right)^{m-m_0}.
\end{equation}
%
%
%
Hence the claimed inequality, as $1\leq m_0\leq m$. Finally, if the regulator is bounded from above then the height is bounded from above as soon as $m\neq m_0$, hence the claimed finiteness, as $m>m_0$ is equivalent with the fact that $A(K)$ is Zariski dense in $A$.
\end{proof}

Theorem \ref{reg height thm intro} follows directly from Proposition \ref{reg height thm V2}, because the set of principally polarized abelian varieties defined over a fixed number field $K$, of fixed dimension $g$ such that $A(K)$ is Zariski dense in $A$ and with regulator and rank bounded from above is also a set of bounded height under Conjecture \ref{LangSilV2}. Note that in view of (\ref{regulators}), one can replace $\Reg_L(A/K)$ by $\Reg(A/K)$ in Theorem \ref{reg height thm intro} because the polarization is principal.

\section{Conclusion}\label{section conclusion}

We generalize here the last section of \cite{Paz14} to abelian varieties, extending the dictionary given in \cite{Hin} as well.
\\

\begin{center}
$\begin{array}{lcccl}
& \mathrm{Number\, field\,} K &  & \mathrm{Abelian\, variety\,} A/K &\\
\\
\mathrm{zeta\, function} & \zeta_K(s) & \leftrightarrow & L(A,s) & L\, \mathrm{function}\\
\mathrm{log \,of \,discriminant}& \log\vert D_K\vert & \leftrightarrow & \hFplus(A) & \mathrm{Faltings \,height}\\
\mathrm{regulator} & R_K & \leftrightarrow & \Reg(A/K) &\mathrm{regulator}\\
\mathrm{class\,number}& h_K &  \leftrightarrow &\vert \Sha(A/K)\vert & \mathrm{Tate}-\mathrm{Shafarevitch \,group}\\
\mathrm{torsion} & (U_K)_{\mathrm{tors}} & \leftrightarrow & (A\times \check{A})(K)_{\mathrm{tors}}& \mathrm{torsion \,of\,} A\, \mathrm{and\, dual\,} \check{A} \\

\mathrm{degree} & d & \leftrightarrow  & g & \mathrm{dimension}\\

\mathrm{max \,subrank\, of\, units}& r_0 & \leftrightarrow & m_0 & \mathrm{max\, rank\, of\, ab.\, subvar. }\\
 \mathrm{relative\, unit\, ranks}& r_K-r_0 & \leftrightarrow & m_K-m_0 & \mathrm{Zariski} \rank \mathrm{of} A(K)\\
 \mathrm{CM \,field} & r_K=r_0 & \leftrightarrow & m_K=m_0& A(K)\;\mathrm{ non\, Z. \,dense}\\
  \mathrm{non-CM \,field} & r_K>r_0 & \leftrightarrow & m_K>m_0& A(K)\;\mathrm{ Zariski \,dense}\\
\\
\end{array}$
\end{center}

\begin{rem}
One could prefer to put in link the property ``$A(K)$ Zariski dense in $A$'' with ``$K$ generated by units''. Let us remark that $A(K)$ Zariski dense is equivalent to $m_K>m_0$, but on the number field side there exists some CM fields $K$ that are generated by units, so $K$ generated by units is not equivalent to $r_K>r_0$. However, regarding the finiteness property obtained from giving an upper bound for the regulator, one may replace the property of being non-CM by the property of being generated by units because there are only finitely many CM fields generated by units with regulator bounded from above. This property of generation, rather than being non-CM, could be seen as a better match to the density property of $A(K)$ on the abelian side.
\end{rem}


\begin{thebibliography}{widest-label}

\bibitem[Aut13]{Aut} \textsc{Autissier, P.}, 
\textit{Un lemme matriciel effectif}.
Mathematische Zeitschrift {\bf 273} (2013), p. 355-361.

\bibitem[BoGu07]{BoGu07} \textsc{Bombieri, E. and Gubler, W.}, 
\textit{Heights in Diophantine Geometry}.
 New Mathematical Monographs, Cambridge University Press 2006 (reprint 2007), {\bf 4} (2007).

\bibitem[Cha00]{Cha} \textsc{Chai, C.-L.}, 
\textit{N\'eron models for semiabelian varieties: congruence and change of base field}.
Asian J. Math. {\bf 4} (2000), 715--736.

\bibitem[CaTa12]{CaTa} \textsc{Cadoret, A. and Tamagawa, A.}, 
\textit{Notes on the torsion conjecture}.
In "Groupes de Galois G\'eom\'etriques et differentiels", P. Boalch and J.-M. Couveignes eds., S\'eminaires et Congr\`es, S.M.F. {\bf 27} (2012), 57--68.

\bibitem[Da93]{Sinn} \textsc{David, S.}, 
\textit{Minorations de hauteurs sur les vari\'et\'es ab\'eliennes}.
Bull. Soc. Math. France {\bf121} (1993), 509--544.

\bibitem[DaPh02]{DavPhi} \textsc{David, S. and Philippon, P.}, 
\textit{Minorations des hauteurs normalis\'ees des sous--vari\'et\'es
              de vari\'et\'es abeliennes. {II}}.
Comment. Math. Helv. {\bf77} (2002), 639--700.

\bibitem[Del85]{Del} \textsc{Deligne, P.}, 
\textit{Preuve des conjectures de Tate et de Shafarevitch}.
S\'eminaire N. Bourbaki {\bf616} (1983-1984), 25--41.

\bibitem[DeMu69]{DeMu69} \textsc{Deligne, P. and Mumford, D.}, 
\textit{The irreducibility of the space of curves of given genus}.
Publ. Math. IHES {\bf36} (1969), 75--109.

\bibitem[Fa83]{Falt} \textsc{Faltings, G.}, 
\textit{Endlichkeitss\"atze f\"ur abelsche {V}ariet\"aten \"uber {Z}ahlk\"orpern}.
Invent. Math. {\bf73} (1983), 349--366.

\bibitem[Fa86]{Falt2} \textsc{Faltings, G.}, 
\textit{Finiteness theorems for abelian varieties over number fields}.
Arithmetic Geometry, Cornell and Silverman (editors), Springer-Verlag {\bf} (1986), 9--27.

\bibitem[GaRe12]{GaRe} \textsc{Gaudron, E. and R\'emond, G.},
\textit{Lemmes de Siegel d'\'evitement}. Acta Arith.
 {\bf 154.2} (2012), 125--136.
 
 \bibitem[GaRe14a]{GaRe2}  \textsc{Gaudron, E. and R\'emond, G.}, 
\textit{Polarisations et isog\'enies}.
Duke Math. J.. {\bf 163} (2014), no. 11, 2057--2108.

 \bibitem[GaRe14b]{GaRe3}  \textsc{Gaudron, E. and R\'emond, G.}, 
\textit{Th\'eor\`eme des p\'eriodes et degr\'es minimaux d'isog\'enies}.
Comment. Math. Helv. {\bf 89} (2014), 343--403.

 \bibitem[GoSha64]{GoSha64}  \textsc{Golod, E. S. and Shafarevich, I. R.}, 
\textit{On the class field tower} (Russian).
Izv. Akad Nauk SSSR Ser. Mat. {\bf 28} (1964), 261--272.

\bibitem[HaNi12]{HaNi} \textsc{Halle, L. H. and Nicaise, J.}, 
\textit{N\'eron models and base change}, Lecture Notes in Math., {\bf 2156} (2016).

\bibitem[Har06]{Hart} \textsc{Hartshorne, R.}, 
\textit{Algebraic Geometry.} Graduate Texts in Mathematics, Springer-Verlag {\bf 52} (2006).

\bibitem[Hin07]{Hin} \textsc{Hindry, M.}, 
\textit{Why is it difficult to compute the Mordell-Weil group?} Diophantine geometry, CRM Series, Ed. Norm., Pisa {\bf 4} (2007), 197--219.

\bibitem[HiPa16]{HiPa16} \textsc{Hindry, M. and Pacheco, A.}, 
\textit{An analogue of the Brauer-Siegel theorem for abelian varieties in positive characteristic}, Mosc. Math. J.  {\bf 16} (2016), 45--93.

\bibitem[Java14]{Java14} \textsc{Javanpeykar, A.},
\textit{Polynomial bounds for Arakelov invariants of Belyi curves}. Algebra and Number Theory {\bf 8.1} (2014), 89--140.

\bibitem[Liu02]{Liu} \textsc{Liu, Q.}, 
\textit{Algebraic Geometry and Arithmetic Curves.} Oxford Graduate Texts in Mathematics, Oxford Science Publications {\bf 6} (2002).

\bibitem[Mai00]{Mai00} \textsc{Maire, C.}, 
\textit{On infinite unramified extensions.} Pacific J. of Math., {\bf 192.1} (2000), 135--142.

\bibitem[Mil08]{Mil} \textsc{Milne, J.-S.}, 
\textit{Abelian varieties.}  www.jmilne.org/math/{\bf 2.0} (2008).

\bibitem[MB89]{MB} \textsc{Moret-Bailly, L.}, 
\textit{La formule de Noether pour les surfaces arithm\'etiques.} Invent. Math., {\bf 98} (1989), 491--498.

\bibitem[Paz10]{Paz2} \textsc{Pazuki, F.}, 
\textit{Remarques sur une conjecture de {L}ang}. Journal de Théorie des Nombres de Bordeaux {\bf 22} no.1 (2010), 161--179.

\bibitem[Paz12]{Paz3} \textsc{Pazuki, F.}, 
\textit{Theta height and Faltings height}. Bull. Soc. Math. France {\bf 140.1} (2012), 19--49.

\bibitem[Paz13]{Paz13} \textsc{Pazuki, F.}, 
\textit{Minoration de la hauteur de N\'eron-Tate sur les surfaces ab\'eliennes}. Manuscripta Math. {\bf 142} (2013), 61--99.

\bibitem[Paz14]{Paz14} \textsc{Pazuki, F.}, 
\textit{Heights and regulators of number fields and elliptic curves}. Publ. Math. Besan\c{c}on {\bf 2014/2} (2014), 47--62.

\bibitem[Paz16a]{Paz16} \textsc{Pazuki, F.}, 
\textit{Northcott property for the regulators of number fields and abelian varieties}. Oberwolfach Rep. {\bf } (2016), to appear.

\bibitem[Paz16b]{Paz16b} \textsc{Pazuki, F.}, 
\textit{Erratum and addendum to ``Heights and regulators of number fields and elliptic curves''}. Publ. Math. Besan\c{c}on {\bf 2016} (2016), to appear.

\bibitem[Phi91]{Phi} \textsc{Philippon, P.}, 
\textit{Sur les hauteurs alternatives I}. Math. Annalen {\bf 289} (1991), 255--283.

\bibitem[Ray85]{Ray} \textsc{Raynaud, M.}, 
\textit{Hauteurs et isog\'enies}. Seminar on arithmetic bundles: the Mordell conjecture (Paris, 1983/84), Ast\'erisque {\bf 127} (1985), 199--234.

\bibitem[R\'em00]{Remond3} \textsc{R{\'e}mond, G.}, 
\textit{D\'ecompte dans une conjecture de {L}ang}.
Invent. Math. {\bf142} (2000), 513--545.

\bibitem[R\'em05]{Remond2} \textsc{R\'emond, G.}, 
\textit{In\'egalit\'e de Vojta g\'en\'eralis\'ee.} Bull. Soc. Math. France {\bf 133.4} (2005), 459--495.

\bibitem[R\'em10]{Remond} \textsc{R\'emond, G.}, 
\textit{Nombre de points rationnels des courbes.} Proc. Lond. Math. Soc. {\bf 101.3} (2010), 759--794.

\bibitem[SeTa68]{SeTa} \textsc{Serre, J.-P. and Tate, J.},
\textit{Good reduction of abelian varieties}. Annals of Math.
 {\bf 88.3} (1968), 492--517.

\bibitem[SGA72]{SGA} \textsc{Grothendieck, A.}, 
\textit{Groupes de monodromie en g\'eom\'etrie alg\'ebrique}. SGA 7.1, Lecture Notes in Mathematics, Springer-Verlag, {\bf 288} (1972).

 \bibitem[SiZa95]{SiZa95} \textsc{Silverberg, A. and Zarhin, Yu.},
\textit{Semistable reduction and torsion subgroups of abelian varieties}.
Ann. Inst. Fourier {\bf45} (1995), 403--420.

\bibitem[Si84b]{Sil3} \textsc{Silverman, J. H.},
\textit{Lower bounds for height functions}.
Duke Math. J. {\bf51} (1984), 395--403.
 
\bibitem[VaWi13]{VaWi13} \textsc{Vaaler, J.D. and Widmer, M.},
\textit{A note on small generators of number fields}.
Diophantine methods, lattices and arithmetic theory of quadratic forms, Contemp. Math. {\bf587} (2013), 214--225.
 
\bibitem[Wil16]{Wil16} \textsc{Wilms, R.}, 
\textit{New explicit formulas for Faltings delta invariant}. preprint arXiv:1605.00847{\bf} (2016).

\end{thebibliography}
\end{document}